\documentclass{article}

\usepackage{amsmath}
\usepackage{amssymb}
\usepackage{amsthm}
\usepackage{url}
\usepackage{graphicx}
\usepackage{epsfig}
\newtheorem{theorem}{Theorem}
\newtheorem{proposition}{Proposition}
\newtheorem{definition}{Definition}
\newtheorem{remark}{Remark}
\newtheorem{lemma}{Lemma}

\begin{document}
\renewcommand{\thefootnote}{}
\footnotetext{Research partially supported by Ministerio de Educaci\'on Grants No: MTM2013-43970-P,  No: PHB2010-0109, Junta de Anadaluc\'\i a Grants No. FQM325, N0. P06-FQM-01642. Minist\'erio de Ci\^encia e Tecnologia, CNPq Proc. No. 303774/2009-6. Minist\'erio de Educa\c{c}\~ao, CAPES/DGU Proc. No. 23038010833/2010-37. }

\title{A class of surfaces related to a problem posed by \'{E}lie Cartan} 
\author{Antonio Mart\'{\i}nez and Pedro Roitman}
\vspace{.1in}

\date{}
\maketitle
\noindent Departamento de Geometr\'\i a y Topolog\'\i a, Universidad de Granada, E-18071 Granada, Spain\\ 
e-mail: amartine@ugr.es

\vspace{.1in}

\noindent Departamento de Matem\'atica, Universidade de Bras\'\i lia, D.F., Brazil 70910-100\\
e-mail: roitman@mat.unb.br 

\begin{abstract}
We introduce a class of surfaces in euclidean space motivated by a problem posed by \'{E}lie Cartan. This class furnishes what seems to be the first examples of pairs of non-congruent surfaces in euclidean space such that, under a diffeomorphism $\Phi$, lines of curvatures are preserved and principal curvatures are switched. We show how to construct such surfaces using holomorphic data and discuss their relation with minimal surfaces. We also prove that if the diffeomorphism $\Phi$ preserves the conformal class of the third fundamental form, then all examples belong to the class of surfaces that we deal with in this work.   
\end{abstract}



\section{Introduction}

In his classic book about exterior differential systems and its applications to differential geometry, \cite{Cartan}, \'{E}. Cartan considered the problem of finding pairs of non-congruent surfaces $M$ and $M^{*}$ in euclidean space and a diffeomorphism $\Phi: M \mapsto M^{*}$ such that $\Phi$ preserves lines of curvature and also the mean and Gaussian curvatures.   

As Cartan remarked, there are two cases to be considered. To distinguish them, we introduce some notation. We denote by $k_{i}$, $i=1,2$, the principal curvatures associated to principal directions $e_{i}$ of $M$, and by $k^{*}_{i}$, $e^{*}_{i}$ the corresponding functions and vector fields on $M^{*}$, where the direction defined by $e_{i}$ is mapped into the direction defined by $e^{*}_{i}$ by the differential $d\Phi$.

Then, since the mean and Gaussian curvatures are preserved we have two possibilities. First case: $k_{1}=k^{*}_1$ and $k_{2}=k^{*}_2$. Second case: $k_{1}=k^{*}_2$ and $k_{2}=k^{*}_1$.

The first case can be rephrased by saying that the shape operator is preserved by $\Phi$, and it is certainly the one that has atracked more attention up to now. For this case, there are several explicit examples that were first discussed by Gambier and Finikoff, \cite{GF} and more recently by Bryant, \cite{Br}, and Ferapontov, \cite{Fer}. A very interesting aspect of this first case is that once a pair of surfaces $M$ and $M^{*}$ is found, there is at least a one (real) parameter family of surfaces that realize the same abstract shape operator. The existence of solutions depending on real parameters motivated Ferapontov to relate the geometric problem to some integrable systems. 

Surprisingly, as far as we know, there seems to be no discussion in the literature about the second case of Cartan's problem (except for his own) and not a single explicit example. Cartan formulates the problem in the framework of exterior differential systems and uses the technique described in \cite{Cartan} to determine the degree of freedom, or arbitrary initial data, we would have for this problem in the real analytic category. 

It is a curious fact that, from Cartan's analysis, the solution of the problem for both cases depends on 4 arbitrary real analytic functions of one variable. With this in perspective, the lack of examples in the second case contrasts with their abundance in the first case.

In this work we exhibit a class of surfaces that provides to our knowledge the first explicit examples of solutions for the second case of Cartan's problem, which, from now on, we will abbreviate by C2. This class of surfaces is defined by a simple geometric property and has some additional features that may capture the attention of geometers, such as, holomorphic representation and a local geometric relation with minimal surfaces in euclidean space. Another interesting aspect of this class is that we can define a local notion of duality, so that surfaces of this class naturally come in pairs. We will show that each such pair is an example for C2.

Besides providing a class of examples for C2, we show that, under the additional hypothesis that the diffeomorphism $\Phi$ is a conformal map with respect to the third fundamental forms of $M$ and $M^{*}=\Phi(M)$, all examples are contained in this class (Theorem 2). This can be considered as a first step in the classification of solutions for C2.

Since the classical concepts used to define this class of surfaces originated in the seminal ideas of A. Ribaucour we have chosen to call them Ribaucour surfaces in honour of this great geometer.

From a historical perspective, as far as we know, Ribaucour surfaces are mentioned en passant by Bianchi in the 1903 edition of the Lezioni, \cite{Bi} page 236, while discussing a special case of Weingarten's method to find surfaces isometric to a given surface. It is curious that Bianchi himself didn't seem to be particularly interested on Ribaucour surfaces and his brief remark about them does not appear in the later editions of the Lezioni. We thank the late and great geometer Bianchi for giving us the opportunity to expose these properties.

In section \ref{RS}, we define Ribaucour surfaces and show how to generate examples using as initial data two holomorphic functions. By reversing the role of these functions, we introduce a notion of duality for Ribaucour surfaces and show some geometric relations between a Ribaucour surface and its dual. In particular, we show that they are examples for C2.

In section \ref{ribmin}, we use the classic Ribaucour sphere congruences to show how Ribaucour surfaces are related to minimal surfaces, we illustrate the theory with examples of Ribaucour surfaces generated using Ennepper's surface and a catenoid. 

In section \ref{conformalcase} we give a proof of theorem \ref{classification} using the relation between Ribaucour and minimal surfaces.

\section{Ribaucour Surfaces and their duals}
\label{RS}

In this section we define Ribaucour surfaces and show a local representation in terms of a pair of holomorphic functions. This representation leads to the definition of the concept of dual Ribaucour surface, which itself is a Ribaucour surface. We then show that a Ribaucour surface and its dual are solutions for C2.

\subsection{Geometric definition and local representation}
 
To define Ribaucour surfaces we need first to recall the classical concept of middle sphere congruence. Let $\Sigma$ be an oriented surface in $\mathbb{R}^3$ and $H$ and $K$ denote, respectively, the  mean and Gaussian curvatures of $\Sigma$. For $p \in \Sigma$ consider the point $M(p)$, called the middle point associated to $p$, along the normal line to $\Sigma$ at $p$ that we get by moving from $p$ a signed distance $\frac{H}{K}$ along the normal line. For a local parametrization $X$ of $\Sigma$ with Gauss map $N$ we have
\[
M(p)=X(p)+\frac{H}{K}N(p). 
\]
The locus $M$ of all middle points associated to the points of $\Sigma$ is called the middle surface associated to $\Sigma$.

The two parameter family of spheres centered at points $M(p)$ of the middle surface and with radius $\left|\frac{H(p)}{K(p)}\right|$ is called the middle sphere congruence associated to $\Sigma$.
      
\begin{definition}
An oriented smooth surface in $\mathbb{R}^3$ is a Ribaucour surface if all its middle spheres intersect a fixed sphere along great circles.
\end{definition}

\begin{remark}
The notion of Ribaucour surface is invariant under homothety and translation in $\mathbb{R}^3$. For this reason, without loss of generality, we will simplify our exposition by only considering the case where the fixed sphere is the unit sphere centered at the origin. 
\end{remark}

For a local parametrization as above, it follows from elementary geometry that a surface $\Sigma$ is a   Ribaucour surface if and only if the following relation holds.
\begin{equation}
\label{relation}
<X,X>+\frac{2H}{K}<X,N>+1=0.
\end{equation}

From (\ref{relation}) it is natural to introduce the support function defined by $\rho=<X,N>$ and recall that, assuming that $K\neq 0$, the immersion $X$ can be written as

\begin{equation}
\label{supportparam}
 X = \nabla \rho + \rho N,
\end{equation} 
where $\nabla$ denotes the gradient with respect to the third fundamental form $III = <dN,dN>$.

We shall make use of the classical expression for $\frac{H}{K}$ in terms of the support function $\rho$,
\begin{equation}
-\frac{2 H}{K} = \Delta \rho + 2 \rho,
\label{hk}
\end{equation}

where $\Delta$ is the Laplace operator with respect to $III$.

\begin{lemma}
\label{difeqrho}
Let $X$ be an immersion as in (\ref{supportparam}). The immersion $X$ defines a   Ribaucour surface if and only if the support function $\rho$ satisfies the following equation.
\begin{equation}
\label{pde}
\rho^2 + \rho \Delta\rho = 1 + |\nabla \rho|^2.
\end{equation}
\end{lemma}
\begin{proof}
Straightforward, just use (\ref{relation}) and (\ref{hk}).
\end{proof}

\begin{remark}
\label{constcurvature}
From a geometric point of view, it is worth mentioning that (\ref{pde}) is satisfied if and only if the quadratic form $\frac{1}{\rho^2}III$ has constant intrinsic curvature equal to 1.
\end{remark}

We will now see that the general solution of (\ref{pde}) defined on a simply connected domain of the unit sphere can be written explicitly in terms of two holomorphic functions. We state this more precisely in the following lemma.

\begin{lemma}
\label{solutionsrepre}
Let $\Omega$ be a regular simply connected domain in the unit sphere $\mathbb{S}^2$. Assume that $\Omega$ is parametrized in the following way
\begin{equation}
\label{gaussmap}
N=(\frac{2\Re{f_{1}(z)}}{1+|f_{1}(z)|^2},\frac{2\Im{f_{1}(z)}}{1+|f_{1}(z)|^2},\frac{|f_{1}(z)|^2-1}{1+|f_{1}(z)|^2}),
\end{equation}
where $f_{1}$ is a holomorphic function defined in a simply connected domain $\Lambda$ of $\mathbb{C}$. Let $f_{2}:\Lambda \longrightarrow \mathbb{C}$ be holomorphic and such that $f'_{2}(z)\neq 0$. The function 
$\rho:\Omega\longrightarrow \mathbb{R}$ defined by
\begin{equation}
\label{repre}
\rho=\frac{|f'_{1}(z)|(1+|f_{2}(z)|^2)}{|f'_{2}(z)|(1+|f_{1}(z)|^2)},
\end{equation}
is a solution of (\ref{pde}). Conversely, every solution $\rho$ of (\ref{pde}) can be represented in the form (\ref{repre}).
\end{lemma}

\begin{proof}

It is easy to chek that $\rho$ defined by (\ref{repre}) is a solution of (\ref{pde}). To prove the converse, we note that from remark \ref{constcurvature} the metric $d\sigma^2=\frac{1}{\rho^2}<dN,dN>$ has constant curvature equal to 1. It is well known that locally such metrics with constant curvature equal to 1 can be written as
\[
d\sigma^2=\frac{4|f'_2(z)|}{(1+|f_2(z)|^2)^2}|dz|^2,
\]
where $f_{2}(z)$ is holomorphic. By the same reason, locally $<dN,dN>$ is written as
\[
<dN,dN>=\frac{4|f'_1(z)|}{(1+|f_1(z)|^2)^2}|dz|^2,
\]
where $f_{2}(z)$ is holomorphic.

It is then clear the $\rho$ can be written as in (\ref{repre}).

\end{proof}

We can now give a local representation for Ribaucour surfaces in terms of holomorphic data.
\begin{proposition}
\label{localrep}
Let $S$ be a   Ribaucour surface. For every point $p\in S$ there is a simply connected neighborhood $V_{p}$ of $p$ in $S$ and holomorphic functions $f_{1}$ and $f_{2}$ defined on a simply connected domain of $\mathbb{C}$ such that $V_{p}$ is parametrized by
\begin{equation}
\label{parametrization}
X=\nabla \rho + \rho N,
\end{equation}
where $N$ and $\rho$ are given respectively by (\ref{gaussmap}) and (\ref{repre}).

Conversely, for every pair of holomorphic functions $f_{1}$ and $f_{2}$ defined on a simply connected domain of $\mathbb{C}$, if we define  $N$ and $\rho$ as in (\ref{gaussmap}), (\ref{repre}) and consider the map $X$ defined by (\ref{parametrization}), then, for points where $X$ is an immersion, its image is a   Ribaucour surface.  
\end{proposition}
\begin{proof}
For a   Ribaucour surface the Gaussian curvature is not zero so the Gauss map is a local diffeomorphism and we may consider a local parametrization of $\mathbb{S}^2$ as in lemma \ref{solutionsrepre} and use this lemma. The converse follows simply from lemmas \ref{difeqrho} and \ref{solutionsrepre}.  
\end{proof}

\subsection{Duality for Ribaucour surfaces and C2}
\label{dualrib}

Before defining a notion of duality for Ribaucour surfaces we show that they are Laguerre isothermic, which means that they admit local parametrizations by lines of curvature that are conformal with respect to the conformal structure defined by the third fundamental form. This fact is related to our discussion in section \ref{conformalcase}  and is also a useful tool to simplify the computations that follow. There has been a recent revival of Laguerre differential geometry and the interested reader may consult \cite{AGL} and references therein.
  
\begin{proposition}
\label{laguerreisothermic}
Ribaucour surfaces are Laguerre isothermic.
\end{proposition}
\begin{proof}
Consider a local parametrization as in proposition \ref{localrep} and let $h_{ij}$, $i,j=1,2$, be the coefficients of the second fundamental form. From \cite{MN} we know that a surface is Laguerre isothermic if and only if there is a function $\psi$ such that
\begin{equation}
\label{laguerrehopf}
\mu=\psi (\frac{h_{22}-h_{11}}{2}+ih_{12}),
\end{equation}
is holomorphic.

We will show that for Ribaucour surfaces we can choose $\psi=\rho^{-1}$ in (\ref{laguerrehopf}) and verify that $\mu$ is holomorhic. It is convenient to use the complex parameter $z$ in our computations. We will write
\[
<dN,dN>=e^{2\tau}|dz|^2,
\]
where $\tau$ is a solution of
\begin{equation}
\label{curvone}
\Delta_{0}\tau+e^{2\tau}=0.
\end{equation}
If we write our immersion as $X=\nabla \rho +\rho N$, it follows that
\[
\frac{h_{22}-h_{11}}{2}+ih_{12}=2<X_{z},N_{z}>.
\]
Using the expression 
\[
\nabla \rho=2e^{-2\tau}(\rho_{z} N_{\bar{z}}+\rho_{\bar{z}}N_z)
\]
and the fact that $z$ is a conformal parameter we obtain
\[
\mu=\frac{2}{\rho}(-2\tau_{z} \rho_{z}+\rho_{zz}).
\]
If we note that equations (\ref{pde}) and (\ref{curvone}) are equivalent respectively to 
\[
\rho^2+4e^{-2\tau}\rho \rho_{z\bar{z}}=1+4e^{-2\tau}\rho_{z} \rho_{\bar{z}},
\]
and
\[
\tau_{z\bar{z}}=-\frac{e^{2\tau}}{4},
\]
then a straightforward computation shows that indeed $\mu_{\bar{z}}=0$.    
\end{proof}

Now we define the local concept of dual Ribaucour surface. By proposition \ref{laguerreisothermic}, given a Ribaucour surface we may consider a local parametrization by lines of curvature and conformal with respect to the third fundamental form and also write it as in (\ref{supportparam}). Now consider $f_{1}$ and $f_{2}$ as in lemma \ref{solutionsrepre} and define the map $N^{*}$ by
\begin{equation}
\label{dualgaussmap}
N^{*}=(\frac{2\Re{f_{2}(z)}}{1+|f_{2}(z)|^2},\frac{2\Im{f_{2}(z)}}{1+|f_{2}(z)|^2},\frac{|f_{2}(z)|^2-1}{1+|f_{2}(z)|^2}),
\end{equation}
and the function $\rho^{*}=\rho^{-1}$. The map $X^{*}$ defined by
\begin{equation}
\label{dualparam}
 X^{*} = \nabla^{*} \rho^{*} + \rho^{*} N^{*},
\end{equation}  
is the local parametrization of the so called dual Ribaucour surface. We now proceed to show that if $X^{*}$ is an immersion then it is a local parametrization of a Ribaucour surface. Moreover, we show that the correspondence between points of these Ribaucour surfaces defined by $X$ and $X^{*}$ preserves lines of curvature and switches principal curvatures. In this way, Ribaucour surfaces and their corresponding duals provide a large class of examples for C2.

\begin{theorem}
Let $X$ and $X^{*}$ be as above. If $X^{*}$ is an immersion then it is a Ribaucour surface. Moreover, the correspondence between $X$ and $X^{*}$ preserves lines of curvature and if we denote by $k_{1}$, $k_{2}$ the principal curvatures of $X$ and by $k^{*}_{1}$, $k^{*}_{2}$ the principal curvatures of $X^{*}$ then $k_{1}=k^{*}_{2}$ and $k_{2}=k^{*}_{1}$.
\end{theorem}
\begin{proof}
From the definition of the dual Ribaucour surface it follows that the third fundamental forms of $X$ and $X^{*}$ are related by 
\begin{equation}
\label{confthird}
III=\frac{III^{*}}{{\rho^{*}}^{2}}
\end{equation}
and by remark \ref{constcurvature}, we conclude that $X^{*}$ is a Ribaucour surface.

To prove the assertion about the relation between the lines of curvature and the principal curvatures we use a conformal parameter $z$ for $III$ and $III^{*}$ such that $X$ is a parametrization by lines of curvature. We may then write $III=e^{2\tau}|dz|^2$ and $III^{*}=e^{2{\tau}^{*}}|dz|^2$. Using (\ref{confthird}) we conclude that ${\tau}^{*}=\tau - \ln{\rho}$ and from this it follows from a simple computation that $\mu$, as in (\ref{laguerrehopf}) with $\psi=\rho^{-1}$, and the analogous function ${\mu}^{*}$ for $X^{*}$ are related by $\mu=-{\mu}^{*}$. From this it follows that
\[
\frac{1}{k_{2}}-\frac{1}{k_{1}}=-(\frac{1}{k^{*}_{2}}-\frac{1}{k^{*}_{1}}).
\]

Finally, a computation shows that $\Delta \rho + 2 \rho={\Delta}^{*} {\rho}^{*} + 2 {\rho}^{*}$ and from (\ref{hk}) it follows that $\frac{H}{K}=\frac{H^{*}}{K^{*}}$. Thus, $k_{1}=k^{*}_{2}$ and $k_{2}=k^{*}_{1}$. 
\end{proof}

For future use in the final section of this work we register how the fundamental forms of a pair of dual Ribaucour surfaces are related. 

\begin{proposition}
Let $M$ and $M^{*}$ be a pair of dual Ribaucour surfaces with fundamental forms given respectively by $(I,II,II)$ and $(I^{*},II^{*},III^{*})$ and let $\rho$, $H$ and $K$ denote respectively the support function, mean and Gaussian curvatures of $M$. Then 
 \begin{eqnarray}
I^{*}&=& \frac{1}{\rho^2} I -\frac{4 H}{K\rho^2} II + \frac{4 H^2 }{(K\rho)^2} III, \label{f1}\\
II^{*}&=& \frac{-1}{\rho^2} II + \frac{ 2H}{K\rho^2}  III, \label{f2}\\
III^{*}&=& \frac{1}{\rho^2} III, \label{f3} 
\end{eqnarray}
\end{proposition}
\begin{proof}
By the definition of duality for Ribaucour surfaces (\ref{f3}) is satisfied. Using a parametrization by lines of curvature we may write the fundamental forms as follows

\begin{eqnarray}
I &= &g_{11} du^2 + g_{22}dv^2, \\
II & = & k_1 g_{11} du^2 + k_2 g_{22}dv^2, \\
III & = & k_1^2 g_{11} du^2 + k_2^2 g_{22} dv^2, \\
I^{*} &= &g^{*}_{11} du^2 + g^{*}_{22}dv^2, \\
II^{*} & = & k_2 g^{*}_{11} du^2 + k_1 g^{*}_{22}dv^2,\\
III^{*} & = & k_2^2 g^{*}_{11} du^2 + k_1^2 g^{*}_{22} dv^2. 
\end{eqnarray}

From (\ref{f3}) and the above expressions for $III$ and $III^{*}$ we deduce that  $g^{*}_{11}=\frac{k^{2}_{1}}{k^{2}_{2}{\rho}^{2}}g_{11}$ and $g^{*}_{22}=\frac{k^{2}_{2}}{k^{2}_{1}{\rho}^{2}}g_{22}$. Using these relations it is a simple matter to check that (\ref{f1}) and (\ref{f2}) also hold. 
\end{proof}
 
\section{Ribaucour surfaces and minimal surfaces}
\label{ribmin}
The aim of this section is to give a brief exposition of the relation between Ribaucour and minimal surfaces that we have learned from Bianchi, \cite{Bi}. This relation appears while Bianchi deals with an example of a general technique created by Weingarten to find isometric surfaces to a given surface. The results in this section are used in the proof of theorem \ref{classification} in section \ref{conformalcase}.    

To show the relation between minimal and Ribaucour surfaces we need to recall the classical concept of Ribaucour sphere congruence. 

\begin{definition}
A Ribaucour sphere congruence is a smooth two parameter family of spheres such that the correspondence between its envelopes preserves lines of curvature.
\end{definition}

Bianchi considered the problem of finding Ribaucour sphere congruences such that one of the envelopes is a minimal surface. He formulated this problem in terms of a system of PDEs. For the reader's convenience we will refer to \cite{keti}, where one can find a contemporary discussion about this topic. 

Suppose that $X(u,v)$ is a local conformal parametrization by lines of curvature for a minimal surface with metric $ds^2=\varphi^2(du^2+dv^2)$ and Gauss map $N$. Consider the following system of PDEs for unknown functions $\Omega$, $\Omega_{1}$, $\Omega_{2}$ and $W$.
\begin{eqnarray}
\Omega_{1,v}&=&\Omega_{2}\frac{\varphi_{u}}{\varphi}, \label{rib1}\\
\Omega_{2,u}&=&\Omega_{1}\frac{\varphi_{v}}{\varphi}, \label{rib2}\\
\Omega_{u}&=&\varphi \Omega_{1}, \label{rib3}\\
\Omega_{v}&=&\varphi \Omega_{2}, \label{rib4}\\
W_{u}&=&\Omega_{1}k_{1}\varphi, \label{rib5}\\
W_{v}&=&\Omega_{2}k_{2}\varphi, \label{rib6}
\end{eqnarray}

To each solution of the above system we can associate a Ribaucour sphere congruence such that $X(u,v)$ is one of the envelopes. The above system is integrable in the Frobenius sense and it has the following first integral.

\begin{equation}
\label{firstintegral}
\Omega_{1}^2+\Omega_{2}^2+W^2-2c\Omega W+c_{2}W+c_{3}\Omega+c_{1}=0,
\end{equation}
where $c_{1}$,$c_{2}$, $c_{3}$ and $c\neq 0$ are constants.

We will show that for a solution of the above system such that $c_{2}=c_{3}=0$ and $c_{1}=1$ in  (\ref{firstintegral})  the surface defined by $W$ as the support function and with $N$ as its Gauss map is a Ribaucour surface. With this purpose in mind, we present a useful relation between the Hessian of $\Omega$ and the fundamental forms of the minimal surface.

\begin{proposition}
Let $X(u,v)$ be a local conformal parametrization by lines of curvature for a minimal surface with metric $ds^2=\varphi^2(du^2+dv^2)$ and $\Omega$ and $W$ part of the solution of the system (\ref{rib1})-(\ref{rib6}) with first integral given by (\ref{firstintegral}). Then ${\nabla}_{I}^{2} \Omega$, the Hessian quadratic form of $\Omega$ with respect to the first fundamental form of the minimal surface, can be expressed as follows.
\begin{equation}
\label{hessianomega}
{\nabla}^{2}_{I} \Omega=(cW-\frac{c_{3}}{2})I+(c\Omega-W-\frac{c_{2}}{2})II,
\end{equation}
where $I$ and $II$ are respectively the first and second fundametal forms of the minimal surface.
\end{proposition}
\begin{proof}
Note that (\ref{rib1}) and (\ref{rib2}) are equivalent to
\[
\frac{\Omega_{uv}}{\varphi}-\frac{\Omega_{u}\varphi_{v}}{\varphi^2}-\frac{\Omega_{v}\varphi_{u}}{\varphi^2}=0,
\]
which is equivalent to $\nabla_{I}^{2} \Omega(\partial_{u},\partial_{v})=0$.

Now if we compute the derivate with respect to $u$ of (\ref{firstintegral}) and use (\ref{rib2}),(\ref{rib3}) and (\ref{rib5}) we end up with

\[
\frac{\Omega_{uu}}{\varphi}-\frac{\Omega_{u}\varphi_{u}}{\varphi^2}+\frac{\Omega_{v}\varphi_{v}}{\varphi^2}=\varphi(cW-\frac{c_{3}}{2})+\varphi k_{1}(c\Omega-W-\frac{c_{2}}{2}),
\]
which we may interpret as
\[
{\nabla}^{2}_{I} \Omega (\partial_{u},\partial_{u})=(cW-\frac{c_{3}}{2})I(\partial_{u},\partial_{u})+(c\Omega-W-\frac{c_{2}}{2})II(\partial_{u},\partial_{u}).
\]

An analogous computation shows that the above relation is also valid if we replace $\partial_{u}$ by $\partial_{v}$ and this yields the desired result.
\end{proof}

The following proposition shows how to construct a Ribaucour surface from a minimal surface and a Ribaucour congruence of spheres.

\begin{proposition}
\label{ribaucourminimal}
Let $X(u,v)$ be a conformal parametrization of a minimal surface and $N$ be its Gauss map. Let $\Omega$, $\Omega_{1}$, $\Omega_{2}$ and $W$ be a solution of the system such that 
\begin{equation}
\label{constant}
\Omega_{1}^2+\Omega_{2}^2+W^2-2c\Omega W+1=0,
\end{equation}
where $c\neq 0$. Then if the map 
\[
Y=\nabla W + WN,
\]  
is an immersion its image is a Ribaucour surface.
\end{proposition}

\begin{proof}
Let $H$ and $K$ denote respectively the mean and Gaussian curvatures of $Y$. We first show that
\[
c\Omega=-\frac{H}{K}.
\]
Since we know that
\[
-2\frac{H}{K}=\Delta W +2W,
\] 
it suffices to compute $\Delta W$. 

Note that from (\ref{rib3})-(\ref{rib6}) it follows that $\tilde{\nabla} \Omega=-{\nabla}W$, where $\tilde{\nabla}$ denotes the gradient operator with respect to the first fundamental form of $Y$. Using this fact and (\ref{hessianomega}), the Hessian quadratic form of $W$ respect to the third fundamental form $III$ is given by,
\begin{equation}
\label{hessianoW}
\nabla^2_{III} W = cW\ II+(c\Omega-W)\ III.
\end{equation}
Then, it follows that 
$$ \Delta W = 2 (c\Omega-W), $$
and we have the following expression for $c\Omega$
$$ c\Omega = \frac{\Delta W + 2 W}{2} = -\frac{H}{K}.$$
Using this fact and (\ref{constant}) we obtain 
\[
|\nabla W|^2+W^2+2\frac{H}{K}W+1=0,
\]
which proves that $Y$ defines a Ribaucour surface.
\end{proof}

For our purposes in the last section it is necessary to write explicitly the relation between the fundamental forms of surfaces that generalizes the construction given above.

\begin{proposition}
\label{fundformsrelation}
Let $(I_{m},II_{m},III_{m})$ and $(I_{r},II_{r},III_{r})$ be respectively the fundamental forms of a minimal surface with Gauss map $N$ and the surface generated from an arbitrary solution of (\ref{rib1})-(\ref{rib6}) with first integral given by (\ref{firstintegral}) such that $W$ is its support function and $N$ its Gauss map. These fundamental forms are related in the following way. 
\begin{eqnarray}
I_{r}&=&a^2I_{m}+2ab II_{m}+b^2 III_{m},\\
II_{r}&=&aII_{m}+bIII_{m}, \\
III_{r}&=&III_{m}, 
\end{eqnarray}
where $a=\frac{c_{3}}{2}-cW$ and $b=\frac{c_{2}}{2}-c\Omega$.
\end{proposition}
\begin{proof}
From (\ref{hessianomega}) and using the same procedure and notation as in the proof of proposition \ref{ribaucourminimal},  we get, 
$$ dY(E) = (\frac{c_{3}}{2}-cW) E -   (\frac{c_{2}}{2}-c\Omega) dN(E), $$ for any tangent vector field $E$, which  let us to easily prove  the desired result. 
\end{proof}

\subsection{Ribaucour surfaces associated to Ennepper's surface and Catenoid}

We illustrate the content of proposition \ref{ribaucourminimal} with two explicit examples.

\subsubsection{A surface associated to Enneper's surface} 

Consider Enneper's surface parametrized as
\[
X(u,v)=(u(1-\frac{u^2}{3}+v^2),-v(1-\frac{v^2}{3}+u^2),u^2-v^2).
\]

The explicit solutions of the system (\ref{rib1})-(\ref{rib6}) are given in (\cite{keti}). It is easy to choose the constants that appear in the general solution in such a way that (\ref{constant}) is satisfied. For instance, if we take
\[
W=\frac{2\cosh{u}}{1+u^2+v^2},
\]
and 
\[
\Omega= (5+u^2+v^2)\cosh{u}+4u\sinh{u}+5\cosh{u},
\]
the corresponding Ribaucour surface is illustrated in figure 1.  

\begin{figure}[ht]
\hspace*{2.0cm}
\includegraphics[scale=0.4]{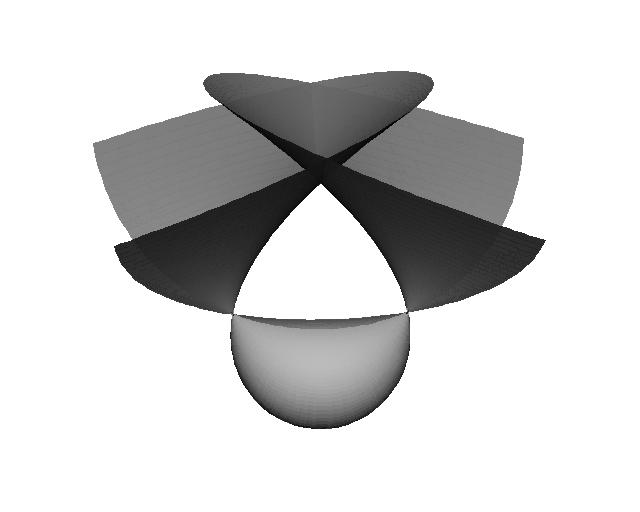} 
\caption{A Ribaucour surface associated to Enneper's surface}
\end{figure}
\subsubsection{A Surface associated to a catenoid}

Consider the parametrized catenoid given by

\[
X(u,v)=(\cosh{v}\cos{u},\cosh{v}\sin{u},v).
\]
As in the previous example, explicit solutions of the system (\ref{rib1})-(\ref{rib6}) are given in (\cite{keti}) and we may choose appropriate constants in such a way that (\ref{constant}) is satisfied. For instance, if we take
\[
W=\frac{1+u^2+v^2}{2\cosh{v}},
\]
and 
\[
\Omega= \frac{(u^2+v^2)\cosh{v}}{2}-2v\sinh{v}+\frac{5}{2}\cosh{v},
\]
the corresponding Ribaucour surface is illustrated in figure 2. 

\begin{figure}[ht]
\hspace*{2.0cm}
\includegraphics[scale=0.4]{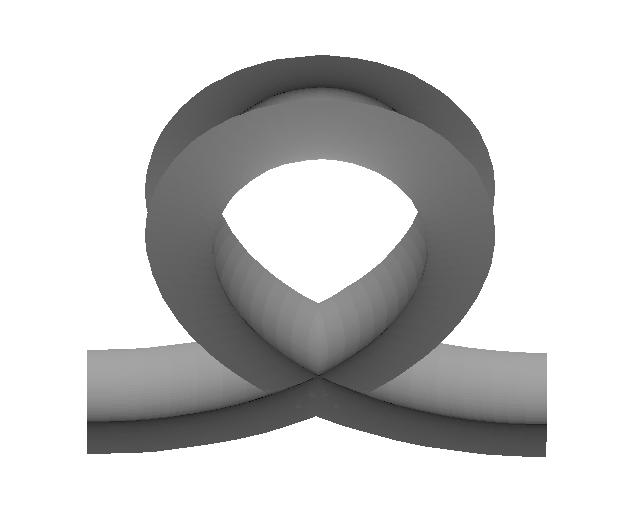} 
\caption{A Ribaucour surface associated to the catenoid}
\end{figure}

\section{Towards a classification of solutions to C2 }
\label{conformalcase}

In subsection \ref{dualrib} we have seen how pairs of dual Ribaucour surfaces are examples of solutions to C2. It is natural to ask how these pairs fit in the set of all solutions to the problem. The present section is devoted to a discussion about this question.

Let $M$ and $M^{*}$ be non congruent and umbilic free surfaces in euclidean space and consider a diffeomorphism $\Phi: M \mapsto M^{*}$ such that $\Phi$ preserves lines of curvature and switches the principal curvatures.  
We shall see, theorem \ref{classification}, that if in addition to the conditions above we impose that $\Phi$ is a conformal map with respect to the third fundamental forms of $M$ and $M^{*}$, then these surfaces are a pair of dual Ribaucour surfaces.

In the discussion that follows we will work with local coordinates $(u,v)$ such that the fundamental forms of $M$ and $M^{*}$ are given by
\begin{eqnarray}
I &= &g_{11} du^2 + g_{22}dv^2, \label{form1}\\
II & = & k_1 g_{11} du^2 + k_2 g_{22}dv^2, \label{form2}\\
III & = & k_1^2 g_{11} du^2 + k_2^2 g_{22} dv^2, \label{form3}\\
I^{*} &= &g^{*}_{11} du^2 + g^{*}_{22}dv^2, \label{form4}\\
II^{*} & = & k_1^* g^{*}_{11} du^2 + k_2^* g^{*}_{22}dv^2,\label{form5}\\
III^{*} & = & (k_1^*)^2 g^{*}_{11} du^2 + (k_2^*)^2 g^{*}_{22} dv^2,\label{form6} 
\end{eqnarray}
and the principal curvatures are related by $k_{1}=k^{*}_2$ and $k_{2}=k^{*}_{1}$.

The Codazzi equations for the pairs of quadratic forms $(I,II)$ and $(I^{*} , II^{*})$ can be written as
\begin{eqnarray*}
(k_1)_v &=& \frac{k_2-k_1}{2} (\log g_{11})_v \\
(k_2)_u &=& \frac{k_1-k_2}{2} (\log g_{22})_u \\
(k_2)_v &=& \frac{k_1-k_2}{2} (\log g^{*}_{11})_v \\
(k_1)_u &=& \frac{k_2-k_1}{2} (\log g^{*}_{22})_u, 
\end{eqnarray*}
and from them we can deduce that
$$  (k_1- k_2)_v = \frac{k_2-k_1}{2} (\log g_{11} g^{*}_{11})_v,$$
$$  (k_1- k_2)_u = \frac{k_2-k_1}{2} (\log g_{22} g^{*}_{22})_u.$$

Thus, there exists functions $\alpha=\alpha(u)$ and $\beta = \beta(v)$, such that 
$$ g_{11} g^{*}_{11}(k_1-k_2)^2=\alpha(u), \qquad g_{22} g^{*}_{22} (k_1-k_2)^2=  \beta(v).$$
By an appropriate change of variables, we can consider local parameters such that (\ref{form1})-(\ref{form6}) still hold and

\begin{equation} g_{11}^{*}=\frac{1}{g_{11} (k_1-k_2)^2}, \qquad  g_{22}^{*}=\frac{1}{g_{22}(k_1-k_2)^2}.
\end{equation}

For our purposes, it is convenient to introduce $$ \Lambda = \frac{k_2-k_1}{k_2 k_1}=2 \sqrt{\frac{(H)^2- K}{(K)^2}},$$ and to rewrite the Codazzi equations in the following form
\begin{eqnarray}
\left(\log (g_{11} k_1^2 \Lambda )\right)_v &= & 2 \Lambda^{-1} \left(\frac{H}{K}\right)_v,\label{c1}\\
\left(\log (g_{22} k_2^2 \Lambda )\right)_u &= &  -2 \Lambda^{-1} \left(\frac{H}{K}\right)_u.\label{c2}
\end{eqnarray}
In addition,  the Gauss equation for $M$ and $M^{*}$ can be written as 
 \begin{eqnarray}
 \sqrt{ g_{11} g_{22}} (k_2-k_1) &=&  \left(\frac{1}{k_1}\right)_{vv} -  \left(\frac{1}{k_1}\right)_{v} (\log \Lambda)_v -   \nonumber\\
 &-& \left(\frac{1}{k_2}\right)_{uu} +\left(\frac{1}{k_2}\right)_{u} (\log \Lambda)_u. \label{g1}
\end{eqnarray}

\begin{eqnarray}
 \frac{1}{\sqrt{ g_{11} g_{22}} (k_1-k_2)} &= & \left(\frac{1}{k_2}\right)_{vv} -  \left(\frac{1}{k_2}\right)_{v} (\log \Lambda)_v -   \nonumber\\
 &-& \left(\frac{1}{k_1}\right)_{uu} + \left(\frac{1}{k_1}\right)_{u} (\log \Lambda)_u. \label{g2}\end{eqnarray}
Subtracting (\ref{g1}) and (\ref{g2}), we have the following equation
\begin{equation}
\label{g3}\Lambda \ \Delta_0\log \Lambda = 2 (\cosh(\log( \sqrt{ g_{11} g_{22}} (k_2-k_1) ))).
\end{equation}

The above equations become much simpler if we impose that the third fundamental forms of $M$ and $M^{*}$ are conformal. We introduce the function $\rho$ and from now on we will consider that $III^{*} = (1/\rho^2) III$. We then have

\begin{equation}
\label{nice}
\rho = g_{11} k_1^2 \Lambda  = g_{22} k_2^2 \Lambda = \sqrt{ g_{11} g_{22}} (k_2-k_1),
\end{equation}
and from (\ref{form1})-(\ref{form6}) it is not difficult to verify that the fundamental forms of $M$ and $M^{*}$ satisfy relations (\ref{f1})-(\ref{f3}). This means that if we can prove that $M$ is a Ribaucour surface then $M^{*}$ has to be its dual Ribaucour surface. We will now continue our discussion to show that this is indeed the case.

Equations (\ref{c1}), (\ref{c2}) and (\ref{g3}) can now be rewritten as
\begin{eqnarray}
\left(\log \rho\right)_u &= & - 2 \Lambda^{-1} \left(\frac{H}{K}\right)_u\label{c3}\\
\left(\log \rho\right)_v &= & 2 \Lambda^{-1} \left(\frac{H}{K}\right)_v\label{c4}\\
\Lambda  \ \Delta_0\log \Lambda &=& \rho + \frac{1}{\rho}. \label{g4}
\end{eqnarray}

Using the preceding relations it is clear that the third fundamental form of $M$ is given by 
$$ III = \frac{\rho}{\Lambda} (du^2 + dv^2), $$
and, since its intrinsic curvature is $1$, we obtain
 
\begin{equation}
\label{gausseq}
 2 = -\frac{\Lambda}{\rho} (\Delta_0(\log \rho - \Delta_0 \log \Lambda) .
\end{equation} 

From above expression we may deduce that (\ref{g4}) is equivalent to
\begin{equation*}
\rho^2 + \rho \Delta \rho = 1 + |\nabla \rho |^2.
\end{equation*}

In other words, $\rho$ is a solution of (\ref{pde}), which is the equation for the support function  that characterizes Ribaucour surfaces with respect to the unit sphere centered at the origin.

So the final step to prove that $M$ is indeed a Ribaucour surface consists in showing that $\rho$ is related to the support function of $M$. In fact, we will show that, up to homothety and translation of $M$, $\rho$ is actually the support function of $M$. To do this we will make use of the relation between Ribaucour and minimal surfaces exposed in section \ref{ribmin}. Before going into the details, it might be worth to say a few words about the way we will proceed.

Our plan goes as follows: we will first define an auxiliary minimal surface, say $\Sigma$, using the fundamental forms of $M$, its gaussian and mean curvatures $K$ and $H$ and also the function $\rho$. Then, we will consider an appropriate solution of the system (\ref{rib1})-(\ref{rib6}) for the minimal surface $\Sigma$ in such a way that the surface considered in proposition \ref{fundformsrelation} is congruent to $M$ and its support function is a constant multiple of $\rho$. From this, we may conclude that $M$ is indeed a Ribaucour surface. With this said, we will now carry out the plan.

Motivated by the inversion of the linear relations between the quadratic forms given in proposition \ref{fundformsrelation}, we define

\begin{eqnarray}
I_{m} &=& \frac{1}{\rho^2}\left(I-\frac{2H}{K}II+\frac{H^2}{K^2}III\right), \\
II_{m} &=& -\frac{1}{\rho}\left(-II+\frac{H}{K} III\right), \\
III_{m}&=&III. \\
\end{eqnarray}

One may check, using (\ref{nice}), that the pair $(I_{m},II_{m})$ can be written in the $(u,v)$ coordinates as
\begin{eqnarray}
I_{m} &=& \frac{\Lambda}{4\rho}(du^2+dv^2), \label{mform1}\\
II_{m} &=& \frac{1}{2}(du^2-dv^2) \label{mform2},
\end{eqnarray}
it is clear that the Codazzi equations for the pair $(I_{m},II_{m})$ are satisfied and also that $I_{m}$ is positive definite. Using (\ref{gausseq}) one can check that the Gauss equation for $(I_{m},II_{m})$ is also satisfied. Thus, by Bonnet's theorem, there exists a unique minimal surface $\Sigma$ up to rigid motion having $I_{m}$ and $II_{m}$ respectively as its first and second fundamental forms.

We now use $\Sigma$ and proposition \ref{fundformsrelation} to construct a surface that is congruent to $M$ and has a multiple of $\rho$ as its support function. This implies immediately that $M$ is a Ribaucour surface since $\rho$ is a solution of (\ref{pde}).

To construct this surface we first define $\Omega=-\frac{H}{K}$, $W=\rho$, $\Omega_{1}=-\sqrt{\frac{\rho}{\Lambda}}(\frac{2H}{K})_{u}$ and $\Omega_{2}=-\sqrt{\frac{\rho}{\Lambda}}(\frac{2H}{K})_{v}$ and verify that they provide solution of the system (\ref{rib1})-(\ref{rib6}). From the above definitions, it is immediate to check that (\ref{rib3}) and (\ref{rib4}) are satisfied. Using (\ref{mform1}), (\ref{mform2}), (\ref{c3}) and (\ref{c4}) we see that (\ref{rib5}) and (\ref{rib6}) are satisfied. Finally, using the expressions obtained by differentiation of (\ref{c3}) with respect to $v$ and (\ref{c4}) with respect to $u$ we can check that (\ref{rib1}) and (\ref{rib2}) are also satisfied.

Recall that a solution of (\ref{rib1})-(\ref{rib6}) admits the first integral (\ref{firstintegral}). But for our particular solution we can determine some constants that appear in (\ref{firstintegral}). From a direct computation as in the proof of proposition \ref{ribaucourminimal} we obtain

\[
\Delta W +2W=2c\Omega-c_{2}.
\]

The expression above combined with the fact that $W$ is a solution of (\ref{pde}) implies that

\[
|\nabla W|^2+W^2-2c\Omega W+c_{2}W+1=0.
\]

We observe that if we change the pair $(\Omega,W)$ into $(\alpha_1 \Omega+ \alpha_{2},\alpha_{1}W)$, where $\alpha_{1}$ and $\alpha_{2}$ are constants, we have a new solution of (\ref{rib1})-(\ref{rib6}). We will choose $\alpha_{1}$ and $\alpha_{2}$ to generate a surface, using proposition \ref{ribaucourminimal}, such that  $I_{r}=I$ and $II_{r}=II$. It is a simple matter to check that it suffices to choose the new $W$ and $\Omega$ as 
\[
W=-\frac{1}{c}\rho , \,\,\, \Omega=-\frac{1}{c}(\frac{-H}{K})+\frac{c_{2}}{2c}.
\]

By construction this new $W$ is the support function of a Ribaucour surface that is congruent to $M$ and so $M$ itself is a Ribaucour surface.
  
Thus, as a consequence of the discussion above we have the following classification result.

\begin{theorem}
\label{classification}
 Let $M$ and $M^{*}$ be non-congruent surfaces with no umbilic points. If there is a diffeomorphism  $\Phi:M\longrightarrow M^{*}$ that preserves lines of curvature and switches the principal curvatures such that $\Phi$ is a conformal map with respect to the conformal structures on $M$ and $M^{*}$ induced by their third fundamental forms, then $M$ and $M^{*}$ are a pair of dual Ribaucour surfaces.  
\end{theorem}

The search for examples to C2 that are not pairs of Ribaucour surface remains an open question.

\end{document}